\documentclass[11pt]{amsart}

\usepackage{a4wide,amsmath,amssymb}
\usepackage[toc,page]{appendix}
\usepackage{hyperref}

\usepackage{url}

\usepackage{courier}

\usepackage{mathtools}

\usepackage{xcolor}

\newtheorem{thm}{Theorem}[section]
\newtheorem*{thmstar}{Theorem}
\newtheorem{lemma}[thm]{Lemma}

\newtheorem{cor}[thm]{Corollary}
\theoremstyle{remark}
\newtheorem{rem}[thm]{Remark}

\theoremstyle{definition}
\newtheorem{defn}[thm]{Definition}
\newtheoremstyle{Claim}{}{}{\itshape}{}{\itshape\bfseries}{:}{ }{#1}
\theoremstyle{Claim}

\newcommand{\R}{\mathbb{R}}

\newcommand{\N}{\mathbb{N}}

\newcommand{\eps}{\varepsilon}

\makeatletter
\@namedef{subjclassname@2020}{\textup{2020} Mathematics Subject Classification}
\makeatother

\theoremstyle{plain}

\def\sideremark#1{\ifvmode\leavevmode\fi\vadjust{
\vbox to0pt{\hbox to 0pt{\hskip\hsize\hskip1em
\vbox{\hsize3cm\tiny\raggedright\pretolerance10000
\noindent #1\hfill}\hss}\vbox to8pt{\vfil}\vss}}}

\begin{document}

\title[Global geometric estimates for the heat equation via duality methods]{Global geometric estimates for the heat equation\\ via duality methods}

\author{Alessandro Goffi}
\address{Dipartimento di Matematica e Informatica ``Ulisse Dini'', Universit\`a degli Studi di Firenze, 
viale G. Morgagni 67/A, 50134 Firenze (Italy)}
\curraddr{}
\email{alessandro.goffi@unifi.it}

\author{Giulio Tralli}
\address{Dipartimento di Matematica e Informatica, Universit\`a degli Studi di Ferrara, Via Machiavelli 30, 44121 Ferrara (Italy)}
\curraddr{}
\email{giulio.tralli@unife.it}

 \thanks{
Both authors are member of the Gruppo Nazionale per l'Analisi Matematica, la Probabilit\`a e le loro Applicazioni (GNAMPA) of the Istituto Nazionale di Alta Matematica (INdAM), and they were partially supported by the INdAM-GNAMPA project 2024. 
}

\date{\today}


\maketitle
\begin{abstract}
We discuss first-order and second-order regularization effects for solutions to the classical heat equation. In particular we propose a global approach to study smoothing effects of Hamilton-Li-Yau type: such approach is nonlinear in spirit and it is based on the Bernstein method and duality techniques \`a la Evans. In a similar way, we also deal with the conservation of geometric properties for the heat flow as initiated by Brascamp-Lieb. In contrast to maximum principle methods based on sup-norm procedures, the integral method we adopt relies on contractivity properties for advection-diffusion equations and it applies to problems with homogeneous Neumann conditions posed equally on bounded and unbounded convex domains under suitable assumptions on their geometry. 
\end{abstract}

\section{Introduction}
\subsection{Overview}
In 1986 P. Li and S.-T. Yau \cite{LiYauActa} proved the following estimate for the heat flow on Riemannian manifolds
\begin{equation}\label{LYintro}
\Delta \log u\geq -\frac{n}{2t},
\end{equation}
$u$ being the solution to the heat equation with respect to the Laplace-Beltrami operator in complete manifolds having nonnegative Ricci curvature, or on compact manifolds with convex boundary where it is prescribed a homogeneous Neumann boundary condition. This inequality is known in the literature as differential Harnack estimate, because it leads, after an integration over a path connecting two different points, to a sharp Harnack inequality for solutions of the heat equation \cite{Davies,LiBook,SchoenYau}. This estimate has its roots in the earlier works \cite{ChengYau,Yau} for elliptic equations, and it was also considered and deepened in a variety of settings, see e.g. \cite{BakryLedoux,ChowHamilton,HanZhang,NiSurvey,Stroock,ZhangIMRN} and the references therein.\\
A further step in this analysis was given by R. Hamilton in \cite{Hamilton}. He proved the following Hessian estimate for the heat flow on a compact Riemannian manifold which is Ricci-parallel and has nonnegative sectional curvatures
\begin{equation}\label{Hintro}
D^2\log u\geq \frac{-\mathbb{I}_n}{2t}.
\end{equation}
Along the way, he also provided the following first-order regularization effect for bounded solutions $0<u\leq A$ of the heat equation on the same domains
\begin{equation}\label{HSZintro}
|D\log u|^2\leq \frac{1}{t}\log\left(\frac{A}{u}\right).
\end{equation}
This bound was then deeply analyzed in \cite{SoupletZhang}, where the authors obtained a local version and addressed several consequences, such as Liouville theorems for entire solutions of the heat equation with sub-exponential decay, see also \cite{Stroock} for a different analysis. \\
A tightly related question is the conservation of log-concavity/convexity from the initial datum for the heat equation. This was first discovered by H. Brascamp and E. Lieb in \cite{BrascampLieb} for the Dirichlet heat flow, where they proved that
\begin{equation}\label{BLintro}
\log u(0)\text{ concave}\implies \log u(t)\text{ concave for $t>0$},
\end{equation}
provided that the (space-dependent) potential of the heat equation is convex.\\ 
These regularity effects can be justified in view of the global behavior of solutions of the so-called viscous Hamilton-Jacobi equation solved by $v=\log u$
\begin{equation}\label{vHJ}
\partial_t v-\Delta v-|Dv|^2=0.
\end{equation}
The nonlinearity appearing in \eqref{vHJ}, i.e. $H(Dv)=-|Dv|^2$, is uniformly concave, and solutions to these equations are known to be semiconvex for positive times, see e.g. \cite{BKL,CS,Douglis,Fleming,K67,Lions82Book}. This property holds even without the second-order diffusive term $-\Delta v$ via the Hopf-Lax formulation \cite{EvansBook,Lions82Book}. Furthermore, such solutions share similar effects to those of the heat equation, as they become instantaneously Lipschitz continuous for positive times \cite{Lions82Book,Lions85aa}. This analogy suggests that these regularizing effects are determined by the nonlinearity $H$ appearing in the equation solved by $\log u$.\\
 In addition, one can justify these smoothing properties in view of a famous analogy among (first-order) Hamilton-Jacobi equations and a regularization procedure of continuous functions introduced in the theory of viscosity solutions to study the well-posedness of fully nonlinear equations \cite{CC,LasryLionsIsrael}. The paper \cite{LasryLionsIsrael} highlighted that one can approximate a continuous function $v$ with the semiconvex function
\[
v^\eps(x)=\sup_{y\in\R^n}\left\{v(y)-\frac{|x-y|^2}{2\eps}\right\}.
\]
If one considers now the Cauchy problem
\begin{equation}\label{firstHJ}
\begin{cases}
\partial_t v(x,t)-\frac{|Dv(x,t)|^2}{2}=0&\text{ in }\R^n\times(0,T),\\
v(0)=u(x)&\text{ in }\R^n,
\end{cases}
\end{equation}
the Hopf-Lax formula shows that
\[
v(x,t)=\sup_{y\in\R^n}\left\{u(y)-\frac{|x-y|^2}{2t}\right\}.
\]
Therefore, if we set $t=\eps$, by the identity $u^\eps(x)=v(x,\eps)$ for $x\in\R^n$, one expects natural first- and second-order regularizing effects as time evolves for the solution of \eqref{firstHJ}.\\
Recall that the inf/sup-convolution procedure is tied up with a classical regularization method by integral convolution. Starting with the Cauchy problem \eqref{firstHJ} with initial condition $v(0)=g(x)$, one can consider the regularized viscous equation
\[
\partial_tv-\eps\Delta v-|Dv|^2=0
\]
and use the so-called Laplace lemma to show that the sup-convolution can be obtained as a singular limit of integral convolutions
\[
\lim_{\eps\to0^+}2\eps\log\left((4\pi \eps t)^{-n/2}\int_{\R^n}e^{-\frac{|x-y|^2}{4\eps t}}e^{\frac{g(x)}{2\eps}}\,dy\right)=\sup_{y\in\R^n}\left\{g(y)-\frac{|x-y|^2}{2t}\right\}.
\]
This is the starting point of the large deviation theory, see Chapter 11 in \cite{Brenier} for full details of the proof.\\

It is also worth mentioning that estimates \eqref{LYintro}-\eqref{Hintro} admit several (other) nonlinear counterparts: a precursory form of the estimate appeared in the context of the porous medium equation and was introduced by D. G. Aronson and P. B\'enilan \cite{AronsonBenilan,Villanietal,VazquezBook}, then for the flow driven by the $p$-Laplacian \cite{EstebanMarcati,EstebanVazquez}, but it also represents a criterium to establish the uniqueness of entropy solutions in the theory of scalar conservation laws, as initiated by O. Ole{\u{\i}}nik \cite{Oleinik}.  \\

It is important to remark here that all the previous global estimates were proved mostly on $\R^n$ or more general complete manifolds (under bounds on their geometry) using nonvariational techniques based on the maximum principle. This technique  of proving one- or two-side estimates by differentiating the equation and applying the maximum principle to the equation solved by the derivatives of the solution is classical and dates back to \cite{Bernstein,BrezisKinderlehrer,Serrin}. Nonetheless, a general treatment of the aforementioned estimates on unbounded settings seems missing in the literature, as it needs generalized maximum principles as in \cite{Yau}. Our aim here is to go beyond classical methods based on the maximum principle, and provide a unifying strategy to prove directly global bounds on general bounded and unbounded domains $\Omega$. Inspired by ideas introduced by L. C. Evans \cite{EvansARMA10}, we adopt a duality approach based on integration by parts techniques. More precisely, we study first- and second-order sup-norm and integral bounds for solutions to \eqref{vHJ} shifting the attention to the formal adjoint of its linearization, that is the backward problem
\begin{equation}\label{fpintro}
-\partial_t \rho-\Delta \rho+\mathrm{div}(b(x,t)\rho)=0\text{ in }\Omega\times(0,\tau).
\end{equation}
This is a natural choice in terms of stochastic control, see \cite{Lions82Book, PorrARMA}. Suitable choices for the terminal datum $\rho(\tau)=\alpha$, the drift $b$, and $L^p$ conservation/contractivity properties of $\rho$ combined with ad hoc integral versions of the Bochner's identity for \eqref{vHJ} provide estimates at different scales for derivatives of the dual solution $v$. As far as equation \eqref{fpintro} is concerned, our main contribution will be a new $L^1$-preserving property for such equations having bounded coefficients posed on unbounded domains with Neumann conditions, see Theorem \ref{well}.

This duality viewpoint was also taken in earlier works in \cite{EvansKAM,LinTadmor}, and more recently in the analysis of nonlinear PDEs and Mean Field Games \cite{surveyMFG,cg20,GomesSM,LL,PorrARMA,TranBook}. Our main results investigate \eqref{LYintro},\eqref{Hintro},\eqref{HSZintro},\eqref{BLintro} on bounded and unbounded domains of the Euclidean space. We prove in particular these estimates when either $\Omega=\R^n$ or $\Omega$ is a suitably regular unbounded convex domain and the problem is equipped with the homogeneous Neumann boundary condition. We also test our method to derive an a priori gradient estimate in $L^p$ spaces, $p\geq2$, that takes the form
\begin{equation}\label{pest}
\|Du(\tau)\|_{L^p(\Omega)}^2\leq \frac{1}{2\tau}\|u(0)\|_{L^{p}(\Omega)}^2,\ 2\leq p\leq\infty.
\end{equation}
We complement our result by proving the upper Laplacian estimate
\begin{equation}\label{reversed}
\Delta \log u\leq \frac{1}{t}\left(n+\frac{7}{2}\log\left(\frac{A}{u}\right)\right).
\end{equation}
Differently from \eqref{LYintro}, the diffusion $\Delta v$ plays a crucial role for the derivation of this bound, see Theorem \ref{HSZ2nd}. In view of the analogy with Hamilton-Jacobi equations, recall that a related phenomenon occurs in the theory of these first-order nonlinear equations: indeed, solutions are known to be $C^{1,1}$ (hence, semiconcave other than semiconvex) for small times, see \cite[Chapter 12]{Lions82Book}. This reversed effect was also pointed out for heat equations with compactly supported initial data in \cite{LeeVazquez} and recently analyzed in \cite{HanZhang}. \\

We conclude the paper by exploring a further application of this duality approach to the preservation of concavity/convexity preserving properties of solutions from the data, as initiated by H. Brascamp and E. Lieb in \cite{BrascampLieb} for the Dirichlet heat flow.  In \cite{BrascampLieb} the proof is based on the explicit representation of the solution through the heat kernel and the application of the Prekopa-Leindler inequality. 
There are nowadays several approaches to establish these geometric properties. One is related to the “macroscopic convexity principle” that involves a globally defined function of two or three variables. This is always accompanied by maximum principle methods and it was the object of the works \cite{Kawohl,Korevaar}, sometimes involving the theory of viscosity solutions \cite{Gigaetal,IshiiLions}. Other methods are based on the so-called constant rank theorems and are related with the so-called ``microscopic'' approach to these geometric properties. They were started almost simultaneously by \cite{CaffarelliFriedman} and \cite{Yauetal}, see also \cite{Weinkove} and the references therein. These results amount to proving that a convex solution of a certain class of elliptic/parabolic PDE has Hessian of constant rank. 
A different viewpoint was undertaken by \cite{AlvarezLasryLions} through the use of viscosity solutions techniques and the convex envelope of the unknown function. This technique was also pushed in \cite{LionsMusiela} to explore necessary and sufficient conditions guaranteeing the convexity property. Propagation of more general notions of convexity and further references can be found in \cite{SalaniMemoirs}.
We mention that only recently a different method through the use of differential games was launched by \cite{KohnSerfaty}, see the survey \cite{LiuSurvey} for recent results in this direction. 

None of these ways of showing the propagation of convexity type properties of linear and nonlinear PDEs use integral techniques, so that all of these results require a control of the data (mainly the potential and/or the initial datum) and their derivatives in sup-norms. Remarkably, other than establishing a new global approach in $\R^n$ towards the macroscopic convexity principle, we weaken some of the hypotheses on the data with respect to the known results in the literature.\\

As a further motivation, we believe that our study could be a useful starting point for a systematic analysis of first- and second-order estimates of Hamilton-Jacobi equations in the presence of Neumann boundary conditions on convex and nonconvex unbounded environments, where few results are available, see e.g. \cite{Lions82Book}.

\subsection{Warm-up: Hamilton matrix Harnack estimates and Li-Yau inequalities}\label{warm}
We start with a quick description of the duality method we will exploit throughout the manuscript. In particular, we begin by testing our integral approach to prove the (semiconvexity) Hamilton matrix estimate \eqref{Hintro} and, as a consequence, we derive the Li-Yau estimate \eqref{LYintro}. As we mentioned, these estimates are well-known, and in case of $\R^n$ they are known to be optimal by looking at the explicit example of the (Gauss-Weierstrass) fundamental solution for the heat equation. The following proof will require a conservation of mass property for the dual Fokker-Planck equation \eqref{fpintro}. The precise assumptions will be clear to the reader in the next sections.

\begin{thmstar}
Fix $T>0$. Let $u$ be a positive solution of
\[
\partial_t u-\Delta u=0\text{ in }\R^n\times(0,T)
\]
such that $u$ has bounded $x$-gradient and $x$-derivatives up to fourth-order in $L^2$. Then the following matrix estimate holds for any $\tau\in(0,T)$
\begin{equation}\label{HamiltonEst}
D^2\log u(\tau)\geq \frac{-\mathbb{I}_n}{2\tau}
\end{equation}
and, in particular,
\begin{equation}\label{LiYauEst}
\Delta(\log u)(\tau)=\frac{\partial_t u}{u}-\frac{|Du|^2}{u^2}\geq -\frac{n}{2\tau}.
\end{equation}
\end{thmstar}
\begin{proof}
Fix $\delta>0$, and let $v=u+\delta$. Then $v$ satisfies \eqref{vHJ}. We denote, for any fixed arbitrary vector $\xi\in\R^n$, $h=\langle D^2v(x,t)\xi,\xi\rangle$ and $z=t^2h$. If we differentiate twice equation \eqref{vHJ} (see Lemma \ref{eqz}-$(ii)$ below with the choice $\chi(t)=t^2$) we find the following equation for the function $z$
\[
\partial_t z-\Delta z-2t^2|D^2v\xi|^2-2\left\langle Dv, Dz\right\rangle=2th\text{ in }\R^n\times(0,T).
\]
We now test the above equation against the solution $\rho$ solving \eqref{fpintro} with the drift $b(x,t)=2Dv$ and with the terminal datum $\rho(\tau)\in C_0^\infty(\R^n)$ with $\|\rho(\tau)\|_{L^1(\R^n)}=1$, $\rho(\tau)\geq0$. We notice that $b(x,t)=2Dv=2\frac{Du}{u+\delta}$ is globally bounded since $u$ is nonnegative and $Du$ is bounded. Hence, such a solution $\rho$ exists and we have the following properties of the solution $\rho$ (cf. Theorem \ref{well} below)
\[
\rho\geq0\text{ a.e. in $\R^n\times(0,\tau)$ and }\int_{\R^n}\rho(t)\,dx=1\text{ for all }t\in[0,\tau].
\]
We are now ready to show the main estimate. By the Cauchy-Schwarz inequality we have
\[
h^2\leq |\xi|^2|D^2v\xi|^2.
\]
This implies by duality and the Young's inequality
\begin{align*}
&\int_{\R^n}z(\tau)\rho_\tau(x)\,dx-2\int_0^\tau\int_{\R^n}t^2|D^2v\xi|^2\rho\,dxdt\\
&=\int_0^\tau\int_{\R^n}2th\rho\,dxdt\geq -2\int_0^\tau\int_{\R^n}t^2h^2\frac{\rho}{|\xi|^2}\,dxdt-\frac12\int_0^\tau\int_{\R^n}|\xi|^2\rho\,dxdt\\
&\geq -2\int_0^\tau\int_{\R^n}t^2|D^2v\xi|^2\rho\,dxdt -\frac12 |\xi|^2 \int_0^\tau 1\,dt.
\end{align*}
By rearranging the terms, we thus have
\begin{equation*}
\int_{\R^n}z(\tau)\rho_\tau(x)\,dx \geq -\tau\frac{|\xi|^2}{2} \qquad\mbox{for any }\rho(\tau)\mbox{ as above,}
\end{equation*}
which implies
\[
z(\tau)=\tau^2\langle D^2v(x,\tau)\xi,\xi\rangle\geq -\tau\frac{|\xi|^2}{2}.
\]
By the arbitrariness of $\xi\in\R^n$ and $\delta>0$, the previous inequality shows \eqref{HamiltonEst}. By taking the trace, \eqref{LiYauEst} follows easily.
\end{proof}

\textit{Outline of the paper}. In Section \ref{sec2} we introduce the geometric setting of the problem and we show the main properties needed for the dual solution $\rho$ of \eqref{fpintro}. Section \ref{sec;regular} is divided in two subSections: subSection \ref{sec;first} starts with a proof of \eqref{pest} and finishes with the logarithmic gradient bound \eqref{HSZintro}, whereas subSection \ref{sec;second} contains a proof of \eqref{LYintro} along with the reversed estimate \eqref{reversed}. Section \ref{sec;conservation} concludes the paper with the conservation of log-concavity/convexity properties of type \eqref{BLintro} for the heat flow on $\R^n$.

\section{Preliminaries: the setting}\label{sec2}

Let us fix the notations we exploit throughout the paper. In addition to this, in the following three separate subSections we discuss the tools needed for the approach we adopt.

\subsection{The geometry of the domain}\label{sec21} Let $\Omega$ be an open and connected set of $\R^n$ with smooth boundary. For $T>0$, we denote
$$
Q_T:=\Omega\times(0,T) \quad\mbox{ and }\quad \Sigma_T:=\partial\Omega\times(0,T),
$$
and $\nu$ stands for the outward normal to the boundary of $\Omega$. The generic point in $Q_T$ is denoted by $(x,t)$, and the gradient, the Hessian, and the Laplacian of a smooth function $u$ with respect to the spatial variable $x$ are respectively denoted by $Du$, $D^2u$, and $\Delta u$. Our analysis focuses on the regularizing properties of the functions $u$ solving
\begin{equation}\label{heat}
\begin{cases}
\partial_t u-\Delta u=0&\text{ in }Q_T,\\
\partial_\nu u=0&\text{ on }\Sigma_T.
\end{cases}
\end{equation}
The techniques developed in this paper work smoothly also in the case $\Omega=\R^n$ and, without further mentioning, we intend that no Neumann condition should appear in such case. The literature around Li-Yau type inequalities we outlined in the Introduction is very rich, and the techniques available have been mainly applied to two scenarios: the case with no boundary ($\Omega=\R^N$, as well as compact manifolds with no boundary) and the bounded case with homogeneous Neumann conditions on convex sets. For this reason, we decided to test our techniques on the case of convex boundaries and to provide new results by allowing the possibility of unbounded sets.\\
To this aim, the following geometric conditions will be assumed in our main results and we will mention explicitly where they are needed:
\begin{align}\label{intcone}
&\Omega\mbox{ satisfies the interior cone condition, i.e. there exists a finite cone such that every}\\ &\mbox{ point in $\Omega$ is the vertex of a cone (congruent to the fixed given cone) contained in $\Omega$;}\notag
\end{align}
\begin{equation}\label{conv}
\Omega\mbox{ is convex.}
\end{equation}
We refer to \cite[Definition 4.3]{Adams} for details about \eqref{intcone}. For our purposes, the assumption \eqref{intcone} yields the extension property for Sobolev spaces and a related apriori estimate for the parabolic equations discussed in subSection \ref{sec23}. On the other hand, the convexity assumption will come into play multiple times whenever we apply the following lemma, which can be seen as the interplay between the homogeneous Neumann condition and the nonnegativity of the principal curvatures of $\partial\Omega$: it is very well-known in the literature and we provide a short proof for the convenience of the reader. 
\begin{lemma}\label{denu}
Let $\Omega$ be a convex domain with smooth boundary $\partial\Omega\neq \emptyset$. Assume that $u$ is $C^2$-smooth around the points on $\partial \Omega$, and it satisfies $\partial_\nu u=0$ on $\partial\Omega$. Then
\[
\partial_\nu|Du|^2\leq0.
\]
\end{lemma}
\begin{proof}
Denote by $II_x^{\partial\Omega}(\cdot,\cdot)$ the second fundamental form of $\partial \Omega$ at a point $x\in\partial\Omega$. Since $Du$ is tangent to $\partial\Omega$ thanks to the Neumann condition, one has
$$
\partial_\nu|Du|^2 (x)=2\left\langle Du , D^2u\, \nu\right\rangle = -2 II_x^{\partial\Omega}(Du,Du).
$$
The convexity ensures that $II_x^{\partial\Omega}(Du,Du)\geq 0$ and concludes the proof.
\end{proof}
We remark that for bounded sets $\Omega$, the assumption \eqref{intcone} is redundant: as a matter of fact the convexity condition \eqref{conv} always implies the interior cone condition for bounded domains. Hence, the hypothesis \eqref{intcone} is helpful for us in providing a uniform control for the geometry of unbounded sets $\Omega$ \emph{at infinity}. It is not difficult to construct classes of unbounded domains satisfying both \eqref{intcone}-\eqref{conv}. Since it is not easy to find examples of unbounded convex sets which do not satisfy \eqref{intcone}, we refer the reader to \cite[Example 3.12]{LeRiVe} for a geometric construction of such an example in $\R^3$.

\vskip 0.3cm

Throughout the paper we will perform various differentiations of the equation, and we recall that we are entitled to do this thanks to the $C^\infty$-regularity of the solutions $u$ to the heat equation. Moreover, as we are dealing with smooth sets $\Omega$, the function $u$ will be smooth up to the boundary points on $\Sigma_T$. As no ambiguity will occur, we shall exploit the notation $u(t)=u(\cdot,t)$. In the paper we consider bounded and/or $L^p$ solutions $u$ to \eqref{heat} which extend continuously up to $t=0$. It is known that this is the case of the solutions of the Cauchy problem starting at $t=0$ respectively from a $L^\infty\cap C$ and/or $L^p$ initial datum $u(0)$. Furthermore, we will require the function $u$ to satisfy the following 
\begin{equation}\label{hypgrad}
Du\in L^\infty(Q_T).
\end{equation}
In our setting, \eqref{hypgrad} is always satisfied if $u$ solves a Cauchy problem starting from a bounded and continuous initial datum at $t=-\eps$ for some positive $\eps$. All such properties are classical in the case $\Omega=\R^n$ and they are treated in several references for bounded $\Omega$ with homogeneous Neumann boundary conditions, as far as the case of unbounded convex $\Omega$ is concerned we refer the reader to the monograph \cite[Part II]{LoreBerto}, where the authors address existence, uniqueness, and gradient bounds within the framework of semigroup theory. 
As we adopt an integral approach, we will perform several integration by parts and we will test the relevant equations against various auxiliary functions depending on the derivatives of the solution $u$. That's why we end this subSection by fixing some integrability requirements for the derivatives of solutions to \eqref{heat} in order to make sense of the computations displayed in the proofs of Section \ref{sec;regular} and Section \ref{sec;conservation}. We use the standard notation $W^{k,q}(\Omega)$ for functions which are in $L^q(\Omega)$ together with their derivatives up to order $k\in\N$. It is known that the solutions of the Cauchy-Neumann problem starting from initial data in $L^2$ belong to $L^2\left(0,T;W^{1,2}(\Omega)\right)$ (see also subSection \ref{sec23} below), but we require something more. To fix the ideas (even if sometimes we need less), we will ask the function $u$ to satisfy the following
\begin{equation}\label{hypderiv}
u\in L^2\left(0,T;W^{4,2}(\Omega)\right).
\end{equation}
These type of higher regularity requirements are standard when the domain $\Omega=\R^n$, see e.g. \cite{DenkHieberPruss}, since in this case one inherits these properties from the initial datum $u(0)$ (i.e. it would be enough to ask $u(0)\in W^{3,2}(\R^n)$). On the other hand, in case of bounded and smooth $\Omega$, the solutions of the Neumann problem are smooth up to the boundary and in particular they satisfy \eqref{hypderiv}. To be more precise, when the bounded domain $\Omega$ is of class $C^2$, solutions of \eqref{heat} are globally in the class $L^2(0,T;H^2(\Omega))$ by Theorem 60 in \cite{LeoniLectures}. Higher regularity then requires more regularity on the domain and the data of the problem, cf. Theorem 70 in \cite{LeoniLectures}. In the case of unbounded sets $\Omega$, classical results establish local $H^2$ regularity via the elliptic problem, see for instance Remark 71 in \cite{LeoniLectures} and \cite{Grisvard}. For unbounded sets which are uniformly regular of class $C^2$, a result of H. Amann \cite[Corollary 15.7]{Amann}, see also \cite[Remark 51.5]{QS}, shows that for $1<q<\infty$ one has maximal $L^q$-regularity estimates
\begin{equation*}
\|u\|_{W^{1,q}(0,T;L^q(\Omega))}+\|u\|_{L^q(0,T;W^{2,q}(\Omega))}\leq C\|u(0)\|_{\mathrm{Tr}},
\end{equation*}
 where $\|\cdot\|_{\mathrm{Tr}}$ is a real interpolation space characterizing the initial datum. 
In the case of general unbounded sets with the Neumann condition, global higher regularity of order $k$ of solutions is studied in \cite[Sections 12 and 25]{Amann} and \cite[Remark 6.7]{AmannCrelle}, see also \cite[Theorems 8.1 and 8.2]{JLLions}, when $\Omega$ is uniformly regular of class $C^k$. Alternatively, Theorem $1_k$ p. 149 of \cite{Browder} established elliptic global estimates for general boundary value problems in the class $W^{2+k,p}(\Omega)$, $k\in\N$, when $\Omega$ is uniformly regular of class $C^{2+k}$. Following then Remark 71 of \cite{LeoniLectures} and using that $\partial_t u=\Delta u$ has the right summability properties, by freezing the time variable we can treat the heat equation as a Poisson equation and conclude the desired integrability uniformly with respect to each $t$. The previous discussion indicates that the assumption \eqref{hypderiv} hides the following condition
\begin{equation*}
\Omega \mbox{ uniformly regular of class }C^4,
\end{equation*}
which is a condition concerning the geometry of smooth unbounded sets $\Omega$ \emph{at infinity} (for the precise definition we refer to \cite{Amann,  Leoni,QS}). The latter condition is, of course, stronger than \eqref{intcone}. However, we point out that the assumptions \eqref{intcone}-\eqref{conv} are exploited in the estimates (in this respect they are quantitative), whereas the assumptions \eqref{hypgrad}-\eqref{hypderiv} are qualitative (in the spirit of a-priori estimates, the sizes of the functional norm involved do not show up in the derivation of the results). In order to get rid of \eqref{hypgrad}-\eqref{hypderiv}, one should perform either a localization or an approximation procedure which is outside of the scope of the global approach carried in this paper and will be the object of a future investigation.

\subsection{Log-solutions and the linearized problem}\label{sec22}
By letting either $v=\log u$ or $v=\log( u + \delta )$ (for some $\delta>0$) where $u$ is a positive solution to \eqref{heat}, we recognize that the pressure function $v$ is a solution to the Neumann problem for the viscous Hamilton-Jacobi equation  
\begin{equation}\label{hjn}
\begin{cases}
\partial_t v-\Delta v -|Dv|^2=0&\text{ in }Q_T,\\
\partial_\nu v=0&\text{ on }\Sigma_T.
\end{cases}
\end{equation}
In the following lemma we explicit the drift-diffusion partial differential equations solved by first and second derivatives of $v$. The proof is completely algebraic and we provide the details for the sake of completeness. 
\begin{lemma}\label{eqz}
Let $v$ be a smooth solution to $\partial_t v-\Delta v -|Dv|^2=0$ in $Q_T$, and let $\chi:(0,T)\to\R^+$ be a $C^1$-function. Then
\begin{itemize}
\item[(i)] the functions $\omega(x,t)=|Dv(x,t)|^2$ and $g(x,t)=\chi(t) \omega$ solve respectively
\[
\partial_t\omega-\Delta \omega+2|D^2v|^2-2\left\langle Dv, D\omega\right\rangle=0\quad\text{ in }Q_T
\]
and
\[
\partial_tg-\Delta g+2\chi(t)|D^2v|^2-2\left\langle Dv, Dg\right\rangle=\chi'(t)\omega\quad\text{ in }Q_T;
\]
\item[(ii)] denoting by $h(x,t)=\left\langle D^2v(x,t)\xi,\xi\right\rangle$ for any fixed $\xi\in\R^n$, the function $z=\chi(t)h$ solves
\[
\partial_t z-\Delta z-2\chi(t)\left|D^2v \xi\right|^2-2\left\langle Dv, Dz\right\rangle=\chi'(t)h\quad\text{ in }Q_T.
\]
As a consequence, $w=\chi(t)\Delta v$ solves
\[
\partial_t w-\Delta w-2\chi(t)|D^2v|^2-2 \left\langle Dv, Dw\right\rangle=\chi'(t)\Delta v\quad\text{ in }Q_T.
\]
\end{itemize}
\end{lemma}
\begin{proof}
The first assertion follows once we plug into the renowned Bochner's identity
\[
\Delta \omega=2|D^2v|^2+2\left\langle Dv, D\left(\Delta v\right)\right\rangle
\]
the equation solved by $\Delta v$ coming from \eqref{hjn}, and we recognize that $2\left\langle Dv, D\left(\partial_t v\right)\right\rangle=\partial_t \omega$. The equation solved by $g$ can be found by multiplying the one for $\omega$ by $\chi$ and using the product rule for the time derivative. This shows the validity of (i).\\
We now discuss (ii). We first find the equation satisfied by $\tilde{h}=v_{x_ix_j}$ by differentiating twice the PDE in \eqref{hjn}, for some fixed $i,j\in\{1,\ldots,n\}$. We thus obtain
\[
\partial_t \tilde{h}-\Delta \tilde{h}-2\sum_{k=1}^n v_{x_ix_k}v_{x_jx_k}-2\sum_{k=1}^nv_{x_k}v_{x_jx_k,x_i}=0.
\]
After having fixed $\xi\in\R^n$, we can take $h$ as in the statement. We then multiply the previous equation solved by $\tilde{h}$ by $\xi_i\xi_j$ and we sum over the indices $i,j$: we discover that $h$ solves
\[
\partial_t h-\Delta h-2\sum_{k=1}^n\left(\sum_{i=1}^n v_{x_kx_i}\xi_i\right)^2-2\left\langle Dv, Dh\right\rangle=0.
\]
We finally multiply the previous equation by $\chi$ and set $z=\chi h$ to find
\[
\partial_t z-\Delta z-2\chi\left|D^2v \xi\right|^2-2\left\langle Dv, Dz\right\rangle=\chi' h.
\]
The equation for $w$ can be derived by rewriting the previous identity by letting $\xi=e_l$ (i.e. the $l$-th vector of the canonical basis) in the definition of $h$, and by summing up over the index $l$.
\end{proof}

\subsection{Well-posedness and $L^1$-stability results for advection-diffusion equations}\label{sec23}

The equations discussed in the previous subSection lead to the the partial differential equations who sit at the center stage of our analysis: the adjoint equation of the linearization of \eqref{hjn}. That is why we now deal with the needed results for the following backward Cauchy-Neumann problem solved by the solution $\rho=\rho(x,t)$ to
\begin{equation}\label{fp}
\begin{cases}
-\partial_t \rho-\Delta \rho+\mathrm{div}(b(x,t)\rho)=0&\text{ in }Q_\tau,\\
\partial_\nu \rho-\rho \left\langle b(x,t), \nu\right\rangle=0&\text{ on }\Sigma_\tau,\\
\rho(\tau)=\rho_\tau&\text{ in }\Omega=\Omega\times\{\tau\}.
\end{cases}
\end{equation}

Here $\tau$ is taken in the interval $(0,T)$, $\mathrm{div}$ stands for the divergence with respect to the $x$-variable, and for our purposes $\rho_\tau$ is always considered a $C_0^\infty(\Omega)$-function (without further mentioning). The key role is played by the vector-valued function $b$ which we assume to be
$$
b\in L^\infty(Q_\tau).
$$
In what follows we shall exploit the properties of the solutions to \eqref{fp} for vector fields $b$ which are also locally smooth up to the boundary of $\Omega$ and the boundary values will be taken continuously, but for the purpose of this section we might think that the boundary values are taken in a weak sense. As a matter of fact, it is convenient for us to work in the Sobolev-space framework, or more precisely in the framework of variational evolution equations. To this aim, we denote by $W'_\Omega$ the dual space of $W^{1,2}(\Omega)$, and by $\left\langle \cdot,\cdot \right\rangle_{W^{1,2}(\Omega)-W'_\Omega}$ the pairing of such duality. We also set
\begin{equation}\label{defW}
W:=\left\{f\in L^2(0,\tau;W^{1,2}(\Omega))\,\mbox{ such that }\,\partial_t f\in L^2(0,\tau;W'_\Omega)\right\}.
\end{equation}
Let us consider weak energy solutions to \eqref{fp} according to the following
\begin{defn}
A function $\rho\in W$ is a weak solution to \eqref{fp} if 
\begin{align*}
&\iint_{Q_\tau}\rho(x,t)\psi(x)\chi'(t)\, dxdt + \iint_{Q_\tau} \left\langle D\rho(x,t)-\rho(x,t)b(x,t),D\psi(x)\right\rangle \chi(t)\, dxdt\\
&=\chi(\tau)\int_{\Omega}\rho_\tau(x)\psi(x)\,dx
\end{align*}
for all $\psi\in W^{1,2}(\Omega)$ and $\chi\in C_0^\infty((0,\tau])$.
\end{defn}
A comprehensive treatment for the existence, uniqueness, and regularity theory of the equation under discussion can be found in the monographs \cite{BKRS,DautrayLions,LSU} under various assumptions on the state space $\Omega$ and on the drift vector field. As far as estimates for the solutions to Cauchy-Neumann problems are concerned we refer the reader to \cite{Davies}, and for unbounded domains with Sobolev-extension properties to the works \cite{ArTer, Daners}. In the following theorem we gather the properties of the solutions $\rho$ we will exploit in the next sections: we stress that most of these results are contained in the reference \cite{Daners}, which cover the case of unbounded domains and time-dependent $L^\infty$-drifts.

\begin{thm}\label{well}
Assume that $\Omega$ satisfies \eqref{intcone}, and let $b\in L^\infty(Q_\tau)$. Then there exists a unique weak energy solution $\rho\in W$ to \eqref{fp}. In addition, if $\rho_\tau\geq0$, then $\rho(t)\geq0$ and we have
$$
\int_{\Omega}\rho(x,t)\,dx=\int_\Omega\rho_\tau(x)\,dx
$$
for $t\in[0,\tau)$.
\end{thm}
\begin{proof}
For the existence and uniqueness of the weak solution $\rho$ belonging in $W$ we refer to \cite[Theorem 2.4]{Daners} (see also Theorem III.5.1 p.170 or Theorem III.5.2 p.171 in \cite{LSU}, and Section XVIII.3.3 in \cite{DautrayLions}). The nonnegativity of $\rho$ under the assumption $\rho_\tau\geq 0$ follows from \cite[Section 8]{Daners} (see also the uniqueness statement in \cite[Theorem III.5.2]{LSU}). We are left with the proof of the conservation of mass property. To this aim, we fix $\rho_\tau\geq 0$ in $C_0^\infty(\Omega)$ with support contained in the ball centered at the origin $B_{R_0}$, and we fix $t\in (0,\tau)$. For $R>1$ let us take a cut-off function $\psi_R(x)=1-\theta_R(x)$ where
$$
\theta_R(x)=\begin{cases}
1\qquad&\text{ if }|x|\geq R,\\
0\qquad&\text{ if }|x|\leq R-1,\\
|x|-R+1&\text{ if }R-1<|x|<R.
\end{cases}
$$
Then, exploiting the fact that $\rho$ solves \eqref{fp} in the weak sense, we have
\begin{align}\label{massdaqui}
&\int_\Omega \psi_R(x) \rho_\tau(x)\,dx - \int_\Omega \psi_R(x) \rho(x,t)\,dx =\int_t^\tau \left\langle \partial_t \rho(s), \psi_R\right\rangle_{W^{1,2}(\Omega)-W'_\Omega}\,ds\\
&=\int_t^\tau\int_\Omega \left\langle D\psi_R(x), D\rho(x,s)-\rho(x,s)b(x,s)\right\rangle\,dxds,\notag
\end{align}
where the first equality is justified by \cite[Theorem 2 of Section XVIII.1.2]{DautrayLions} (keeping in mind that $\rho\in C(0,\tau;L^{2}(\Omega))$).
Since $\psi_R$ is converging pointwise to the constant $1$ and it is monotonically increasing with respect to $R$, we easily obtain by Beppo Levi's theorem that
$$\int_\Omega \psi_R(x) \rho_\tau(x)\,dx\longrightarrow \int_\Omega\rho_\tau(x)\,dx<\infty \,\,\,\mbox{ as }R\to \infty$$
and
$$\int_\Omega \psi_R(x) \rho(x,t)\,dx\longrightarrow \int_\Omega\rho(x,t)\,dx \,\,\,\mbox{ as }R\to \infty.$$
Hence, the proof will be complete once we show that the right hand side of \eqref{massdaqui} vanishes as $R\to \infty$. In order to do so, we recall the estimates proved in \cite[Section 7]{Daners} (see the case $d=N$ in \cite[Section 6 and (7.2)]{Daners}) which ensure that $\rho\in L^\infty(0,\tau;L^1(\Omega))$ and $\rho\in L^\infty(0,\tau;L^2(\Omega))$: more precisely, we have the existence of positive constants $c_1(\tau)$ and $c_2(\tau)$ such that
\begin{equation}\label{unifdan}
\|\rho(\sigma)\|_{L^1(\Omega)}\leq c_1(\tau)\|\rho_\tau\|_{L^1(\Omega)}\mbox{ and }\|\rho(\sigma)\|^2_{L^2(\Omega)}\leq c_2(\tau)\|\rho_\tau\|^2_{L^2(\Omega)}\mbox{ for every }\sigma\in(0,\tau).
\end{equation}
Since $b\in L^\infty$, $\|D\psi_R\|_\infty\leq 1$, and $D\psi_R(x)\equiv 0$ in $B_{R-1}$, we obtain
$$
\left |\int_t^\tau\int_\Omega \rho(x,s)\left\langle D\psi_R(x), b(x,s)\right\rangle\,dxds \right|\leq \|b\|_{\infty} \int_t^\tau\int_{\Omega\smallsetminus B_{R-1}} \rho(x,s)\,dxds\longrightarrow 0
$$
as $R\to \infty$, where the limiting behavior is a consequence of the first inequality in \eqref{unifdan} and of a direct application of the dominated convergence theorem. On the other hand, we have
\begin{align*}
\left |\int_t^\tau\int_\Omega \left\langle D\psi_R(x), D\rho(x,s)\right\rangle\,dxds \right|^2&\leq  \int_t^\tau\int_\Omega |D\psi_R(x)|^2 \,dxds \int_t^\tau\int_{\Omega \smallsetminus B_{R-1}} |D\rho(x,s)|^2 \,dxds\\
&\leq \tau |B_1| R^n \int_t^\tau\int_{\Omega \smallsetminus B_{R-1}} |D\rho(x,s)|^2 \,dxds.
\end{align*}
We claim that
\begin{equation}\label{claimiamo}
R^n \int_t^\tau\int_{\Omega \smallsetminus B_{R}} |D\rho(x,s)|^2 \,dxds\longrightarrow 0 \,\,\,\mbox{ as }R\to \infty.
\end{equation}
It is clear from the above discussion that the validity of the claim \eqref{claimiamo} will finish the proof of the desired statement. For an arbitrary $\sigma\in(0,\tau)$, let us introduce the function
$$
H_R(\sigma):=\int_{\Omega \smallsetminus B_{R}}\rho^2(x,\sigma)\, dx+\int_\sigma^\tau\int_{\Omega \smallsetminus B_{R}} |D\rho(x,s)|^2 \,dxds.
$$
We notice that $H_R(\sigma)$ is decreasing as a function of $R$ and thus
$$
H_R(\sigma)\leq H_0(\sigma):=\int_{\Omega}\rho^2(x,\sigma)\, dx+\int_\sigma^\tau\int_{\Omega } |D\rho(x,s)|^2 \,dxds.
$$
Since we are allowed to test the solution $\rho$ against itself in the weak formulation (see, e.g., the proof of \cite[Lemma 3.3]{Daners}) we get the identity
$$
\frac{1}{2}\int_{\Omega}\rho_\tau^2(x)\, dx-\frac{1}{2}\int_{\Omega}\rho^2(x,\sigma)\, dx=\int_\sigma^\tau\int_{\Omega } \left\langle D\rho(x,s), D\rho(x,s)-\rho(x,s)b(x,s)\right\rangle \,dxds,
$$
from which we deduce that
\begin{align*}
\frac{1}{2}H_0(\sigma)&=\frac{1}{2}\int_{\Omega}\rho_\tau^2(x)\, dx-\frac{1}{2}\int_\sigma^\tau\int_{\Omega } |D\rho(x,s)|^2 \,dxds+\int_\sigma^\tau\int_{\Omega } \rho(x,s)\left\langle D\rho(x,s),b(x,s)\right\rangle \,dxds\\
&\leq \frac{1}{2}\int_{\Omega}\rho_\tau^2(x)\, dx-\frac{1}{2}\int_\sigma^\tau\int_{\Omega } |D\rho(x,s)|^2 \,dxds+\|b\|_\infty\int_\sigma^\tau\int_{\Omega } \rho(x,s)|D\rho(x,s)| \,dxds\\
&\leq\frac{1}{2}\int_{\Omega}\rho_\tau^2(x)\, dx+\frac{1}{2}\|b\|^2_\infty\int_\sigma^\tau\int_{\Omega } \rho^2(x,s) \,dxds,
\end{align*}
where in the last step we used Young's inequality. Hence, the second inequality in \eqref{unifdan} yields
\begin{equation}\label{bzero}
H_0(\sigma)\leq (1+\|b\|^2_\infty c_2(\tau)\tau)\|\rho_\tau\|^2_{L^2(\Omega)}=:C\|\rho_\tau\|^2_{L^2(\Omega)} \mbox{ for every }\sigma\in(0,\tau).
\end{equation}
If we now take $R\geq R_0+1$ and we test the equation for $\rho$ with $\theta_R \rho$ we obtain 
\begin{align*}
0=\frac{1}{2}\int_{\Omega}\theta_R(x)\rho_\tau^2(x)\, dx=\frac{1}{2}\int_{\Omega}&\theta_R(x)\rho^2(x,\sigma)\, dx + \int_\sigma^\tau\int_{\Omega } \left\langle D\left(\theta_R(x)\rho(x,s)\right), D\rho(x,s)\right\rangle \,dxds +\\
&- \int_\sigma^\tau\int_{\Omega } \rho(x,s) \left\langle D\left(\theta_R(x)\rho(x,s)\right),b(x,s)\right\rangle \,dxds,
\end{align*}
which implies that
\begin{align*}
&\frac{1}{2}\int_{\Omega\smallsetminus B_R}\rho^2(x,\sigma)\, dx + \int_\sigma^\tau\int_{\Omega\smallsetminus B_R } |D\rho(x,s)|^2 \,dxds \\
&\leq\frac{1}{2}\int_{\Omega}\theta_R(x)\rho^2(x,\sigma)\, dx + \int_\sigma^\tau\int_{\Omega } \theta_R(x) |D\rho(x,s)|^2 \,dxds\\
&=\int_\sigma^\tau\int_{\Omega\smallsetminus B_{R-1} } \rho(x,s)\left\langle \theta_R(x)b(x,s)-D\theta_R(x), D\rho(x,s)\right\rangle \,dxds \\
&+ \int_\sigma^\tau\int_{\Omega\smallsetminus B_{R-1} } \rho^2(x,s) \left\langle D\theta_R(x),b(x,s)\right\rangle \,dxds\\
&\leq (1+\|b\|_\infty)\int_\sigma^\tau\int_{\Omega\smallsetminus B_{R-1} } \rho(x,s)|D\rho(x,s)| \,dxds + \|b\|_\infty \int_\sigma^\tau\int_{\Omega\smallsetminus B_{R-1}} \rho^2(x,s) \,dxds.
\end{align*}
Recalling the definition of $H_R$ and using Young's inequality, we deduce that for any $\delta>0$
\begin{align}\label{stimdelta}
H_R(\sigma)&\leq \delta (1+\|b\|_\infty) \int_\sigma^\tau\int_{\Omega\smallsetminus B_{R-1} } |D\rho(x,s)|^2 \,dxds + \notag\\
&+\left(\frac{1+\|b\|_\infty}{\delta}+ 2\|b\|_\infty\right)\int_\sigma^\tau\int_{\Omega\smallsetminus B_{R-1}} \rho^2(x,s) \,dxds \notag\\
&\leq \delta (1+\|b\|_\infty)H_{R-1}(\sigma) + \left(\frac{1+\|b\|_\infty}{\delta}+ 2\|b\|_\infty\right)\int_\sigma^\tau H_{R-1}(s)\,ds.
\end{align}
The estimate \eqref{stimdelta}, together with \eqref{bzero}, implies the following: there exists a positive constant $c(\tau)$ such that the following bound holds true
\begin{align}\label{wgus}
&H_R(\sigma)\leq \frac{c^k(\tau) (\tau-\sigma)^{\frac{k}{2}}}{\sqrt{(k+1)!}} \|\rho_\tau\|^2_{L^2(\Omega)}\\
\mbox{ for every }\sigma\in (0,\tau)&\mbox{ for every }R\in [R_0+k,R_0+k+1)\mbox{ for every } k\in\N.\notag
\end{align}
As a matter of fact, one can argue by induction over $k\in\N$ by showing \eqref{wgus} with the choice $c(\tau)=2\sqrt{2}\left(1+(1+\sqrt{\tau})\|b\|_\infty\right)\max\{1,C\}$ where $C$ is the constant in \eqref{bzero}. The case $k=1$ comes from plugging \eqref{bzero} into the estimate \eqref{stimdelta} with $\delta=\sqrt{\tau-\sigma}$ which yields, for $R\geq R_0+1$,
\begin{align*}
H_R(\sigma)&\leq H_{R_0+1}(\sigma)\leq 2C\sqrt{\tau-\sigma}\|\rho_\tau\|^2_{L^2(\Omega)}\left[1+\|b\|_\infty+ \|b\|_\infty\sqrt{\tau-\sigma}\right]\\
&\leq 2C\left(1+(1+\sqrt{\tau})\|b\|_\infty\right)\sqrt{\tau-\sigma}\|\rho_\tau\|^2_{L^2(\Omega)}.
\end{align*}
On the other hand, if we assume that \eqref{wgus} holds for some $k\in\N$, in order to show \eqref{wgus} for $k+1$ we can plug, for $R\geq R_0+k+1$, the inductive hypothesis into the estimate \eqref{stimdelta} with $\delta=\frac{\sqrt{2}\sqrt{\tau-\sigma}}{\sqrt{k+2}}$. This shows \eqref{wgus}. From \eqref{wgus} we easily deduce
$$
H_R(t)\leq \frac{c^k(\tau) (\max\{1,\sqrt{\tau}\})^{k}}{\sqrt{(k+1)!}} \|\rho_\tau\|^2_{L^2(\Omega)} \mbox{ for every $k\in\N$ such that }R-R_0-1<k\leq R-R_0,
$$
which implies
$$
R^n H_R(t) \longrightarrow 0 \qquad\mbox{ as }R\to\infty.
$$
The previous limiting behavior concludes the proof of the claim \eqref{claimiamo}, and therefore the proof of the theorem.
\end{proof}

One can notice that the focus of the previous proof is on the conservation of mass property. We stress that such property is straightforward in the case of bounded $\Omega$ (in this scenario it is enough to use the test function identically equal to 1). On the other hand, in the case $\Omega=\R^n$ one can follow different paths such as the uniqueness of the constant solution $1$ for the adjoint equation (see the treatments in \cite{DK,PorEid} and also \cite{LanconelliPascucci} for a Gaussian-bound approach). For the general case of unbounded $\Omega$ with Neumann conditions, the strategy we follow in the above proof is heavily inspired by \cite[Corollary 1]{Gus}, but it exploits in a crucial way the uniform estimates in \cite{Daners}.

\begin{rem}\label{contr}
In the sequel, what we really need is that
$$\int_\Omega\rho_\tau(x)\,dx\leq 1\,\,\,\, \Longrightarrow \,\,\,\, \int_\Omega\rho(x,t)\,dx\leq 1\,\,\,\,\mbox{ for all $t\in[0,\tau)$.}$$ This is in general a consequence of $L^1$ contraction estimates, cf. e.g. \cite[estimate (3.12)]{PorrUMI} or Proposition 3.7 in \cite{PorrARMA} (note that they hold under much weaker integrability conditions on the velocity field). The $L^p$ contractivity is in general false for equations with divergence-type terms (the estimate in general depends on appropriate norms of the drift), while it is true for the heat equation. More precisely, for any $p\geq 1$, one has
\begin{align}\label{Lpcontr}
\mbox{if $\rho$ is the weak energy solution to \eqref{fp}} & \mbox{ with $b\equiv 0$ and $\rho_\tau\geq 0$, then}\notag\\
\int_\Omega\rho^p(x,t)\,dx \leq \int_\Omega\rho^p_\tau(x)\,dx & \mbox{ for all $t\in[0,\tau)$.}
\end{align}
The property \eqref{Lpcontr} will be exploited in Theorem \ref{lip0} below, and we notice that the assumption \eqref{intcone} is not needed for the validity of this property. A proof of the fact that the linear map $\rho_\tau\mapsto \rho(t)$ is a contraction in $L^p(\Omega)$ for the Cauchy-Neumann solutions to the heat equation can be in fact found in \cite[Corollary 1]{Gus} for $p=1$ (or by reviewing the above proof), it can be showed by maximum principle (or by duality) for $p=\infty$, and by making use of the Riesz-Thorin interpolation theorem for generic $p$ (alternatively one can invoke \cite[Proposition 7.4 with $\delta_0=0$]{Daners} under the right assumptions on $\Omega$).\\
The above proof of the conservation of mass holds for more general equations of the form $-\partial_t\rho-\mathrm{div}(A(x,t)D\rho)+\mathrm{div}(b(x,t)\rho)=0$ with $A$ uniform elliptic with bounded and measurable entries. We also expect the validity of the conservation of mass property under weaker assumptions on the drift, at least in the whole space $\R^n$. For instance, one might assume the drift $b$ to be locally in the so-called Aronson-Serrin-Ladyzhenskaya regime $L^Q_t(L^P_x)$ with $P,Q$ satisfying
\[
\frac{n}{2P}+\frac{1}{Q}<\frac12,
\]
cf, \cite[p.169]{PorEid}, combined with suitable conditions at infinity, cf. \cite{DK}. We are not aware of similar studies for problems on unbounded domains equipped with the Neumann condition.
\end{rem}

\section{Regularizing effects}\label{sec;regular}
\subsection{First-order smoothing bounds}\label{sec;first}
As a starting point, we consider the following classical estimate
\[
\|Du(t)\|_{L^\infty(\R^n)}\leq \frac{C_n}{\sqrt{t}}\|u(0)\|_{L^\infty(\R^n)},\ C_n=\frac{\Gamma\left(\frac{n+1}{2}\right)}{\Gamma\left(\frac{n}{2}\right)},
\]
which is a consequence of the explicit computation of the gradient of the fundamental solution of the heat equation in $\R^n$. In our first result we look for similar $L^\infty(\Omega)-W^{1,\infty}(\Omega)$ bounds. The main difference will be a kind of oscillation dependence on the right-hand side, cf. \eqref{limcac}, and the obtainment of a constant independent of the dimension $n$. We shall derive here a general regularizing effect $L^p(\Omega)-W^{1,p}(\Omega)$, $2\leq p \leq\infty$, by a duality version of the Bernstein technique: this will be based on $L^p$ contractivity properties of backward heat equations, as the reader should keep in mind \eqref{Lpcontr} in Remark \ref{contr}. 
To this aim, we fix the following notation 
\begin{equation}\label{defFinf}
\mathcal{F}^+_\infty(\Omega):=\left\{\varphi\in C_0^\infty(\Omega)\,:\, \varphi\geq 0,\, \|\varphi\|_{L^1(\Omega)}=1\right\}
\end{equation}
and, for any $p\geq 2$,
\begin{equation}\label{defFp}
\mathcal{F}^+_p(\Omega):=\left\{\varphi\in C_0^\infty(\Omega)\,:\, \varphi\geq 0,\, \|\varphi\|_{L^{(p/2)'}(\Omega)}=1\right\},
\end{equation}
where we intend $1'=+\infty$.

\begin{thm}\label{lip0}
Suppose that $\Omega$ satisfies \eqref{conv}. Let $u$ be a solution to \eqref{heat} in $Q_T$ such that $u\in L^2\cap L^\infty$ and it satisfies \eqref{hypgrad}-\eqref{hypderiv}. Then we have the estimate
\begin{equation}\label{limcac}
|Du(x,\tau)|^2\leq \frac{1}{2\tau}\left(\sup_{y\in\Omega} u^2(y,0)-u^2(x,\tau)\right),\ \mbox{ for any }x\in\Omega\mbox{ and for any }\tau\in (0,T).
\end{equation}
Moreover, for any $p\in [2,\infty]$ and $\tau\in (0,T)$, we have
\begin{equation}\label{2ndA}
\tau\|Du(\tau)\|_{L^p(\Omega)}^2\leq \frac{1}{2}\|u(0)\|_{L^p(\Omega)}^2
\end{equation}
and
\begin{equation}\label{2ndB}
\sup_{\mu_\tau\in \mathcal{F}^+_p(\Omega)}\iint_{Q_\tau}2t|D^2u(x,t)|^2\mu(x,t)\,dxdt\leq \frac{1}{2}\|u(0)\|_{L^p(\Omega)}^2,
\end{equation}
where $\mu$ is the solution of the backward heat equation with homogeneous Neumann condition starting from the datum $\mu(\tau)=\mu_\tau$.
\end{thm}
\begin{proof}The idea is based on a global integral version of the Bernstein method. By putting together the identity
\[
\partial_t(u^2)=\Delta(u^2)-2|Du|^2
\]
with Bochner's identity, we find that the function $G_K(x,t)=t|Du(x,t)|^2+Ku^2(x,t)$, $K>0$ to be fixed, solves
\begin{equation}\label{eqK}
\partial_tG-\Delta G+2t|D^2u|^2= (1-2K)|Du|^2.
\end{equation}
We now choose $K=\frac12$ and obtain that $G=G_{\frac{1}{2}}$ solves
\[
\partial_tG-\Delta G+2t|D^2u|^2=0.
\]
We stress that the function $G$ belongs to $W$ since both the terms $t|Du(x,t)|^2$ and $u^2(x,t)$ are in $L^2(0,T;W^{1,2}(\Omega))$ and their $t$-derivatives are in $L^2(Q_T)$ (this is a consequence of the assumptions on the function $u$ which imply $u\in L^\infty(Q_T)\cap L^2(0,T;W^{3,2}(\Omega))$ and $Du\in L^\infty(Q_T)$). Let us now fix $\tau\in(0,T)$ and $p\in [2,\infty]$. On $Q_\tau$ we can test the previous equation by the nonnegative solution $\mu$ of the backward heat equation (i.e. \eqref{fp} with $b=0$) with $\mu(\tau)=\mu_\tau\in \mathcal{F}^+_p(\Omega)$ and we find the integral inequality
\begin{equation}\label{intineq}
\int_\Omega G(x,\tau)\mu_\tau(x)\,dx+ 2\iint_{Q_\tau}t|D^2u|^2\mu\,dxdt\leq\int_{\Omega}G(x,0)\mu(x,0)\,dx.
\end{equation}
We have exploited here that, when $\Omega$ is a convex domain, we have $\partial_\nu G\leq0$ on $\Sigma$ by Lemma \ref{denu} and by the homogeneous Neumann condition. Note now that
\begin{align*}
\int_{\Omega}G(x,0)\mu(x,0)\,dx&=\frac12\int_{\Omega}u^2(x,0)\mu(x,0)\,dx\\
&\leq \frac12\|u^2(0)\|_{L^{\frac{p}{2}}(\Omega)}\|\mu(0)\|_{L^{(\frac{p}{2})'}(\Omega)}\leq \frac12\|u^2(0)\|_{L^{\frac{p}{2}}(\Omega)}\|\mu(\tau)\|_{L^{(\frac{p}{2})'}(\Omega)}
\end{align*}
since by \eqref{Lpcontr} $\|\mu(0)\|_{q}\leq \|\mu_\tau\|_{q}$ for any $q\geq1$ (and the equality occurs when $q=1$). By plugging the previous inequality into \eqref{intineq} and using $\|\mu_\tau\|_{L^{(\frac{p}{2})'}(\Omega)}=1$, we obtain
\[
\int_\Omega \left(\tau|Du(x,\tau)|^2+\frac12u^2(x,\tau)\right)\mu_\tau(x)\,dx+2\iint_{Q_\tau}t|D^2u|^2\mu\,dxdt\leq \frac12\|u(0)\|^2_{L^{p}(\Omega)}.
\]
Hence we infer the validity of the inequalities
$$
\tau\int_\Omega |Du(x,\tau)|^2 \mu_\tau(x)\, dx\leq \frac12\|u(0)\|^2_{L^{p}(\Omega)}
$$
and
$$
\iint_{Q_\tau}2t|D^2u(x,t)|^2\mu(x,t)\,dxdt\leq \frac12\|u(0)\|^2_{L^{p}(\Omega)}
$$
from which the estimates \eqref{2ndA} and (respectively) \eqref{2ndB} follow by passing to the supremum over $\mu_\tau\in \mathcal{F}^+_p(\Omega)$. In the limiting case $p=\infty$ we get the refinement stated in \eqref{limcac}: as a matter of fact, from \eqref{intineq} and by using the conservation of mass $\|\mu(0)\|_{L^1}= \|\mu_\tau\|_{L^1}=1$, one has
\begin{align*}
\int_\Omega \left(\tau|Du(x,\tau)|^2+\frac12u^2(x,\tau)\right)\mu_\tau(x)\,dx&=\int_\Omega G(x,\tau)\mu_\tau(x)\,dx\\
&\leq \|\mu_\tau\|_{L^1(\Omega)}\|G(x,0)\|_{L^\infty(\Omega)}= \frac12\sup_{y\in\Omega}u^2(y,0)
\end{align*}
for every $\mu_\tau\in\mathcal{F}^+_\infty(\Omega)$. This implies the desired \eqref{limcac}.
\end{proof}

\begin{rem}
A result similar to \eqref{limcac} can be found in Theorem 3 of \cite{LiuPJM} for closed manifolds by the maximum principle, while related gradient estimates depending on the oscillation of the initial datum $u(0)$ were studied in \cite{AndrewsClutterbuck} by doubling of variables techniques in a periodic setting, cf. \cite{AndrewsSurvey} and the references therein for other results.
\end{rem}

\begin{rem}
Let us note that \eqref{limcac} (or \eqref{2ndA} with $p=\infty$) implies
\[
\|Du(\tau)\|_{L^\infty(\R^n)}\leq \frac{1}{\sqrt{2\tau}}\|u(0)\|_{L^\infty(\R^n)}.
\]
In particular, differently from the estimate with constant $C_n$ recalled at the beginning of the Section, the constant is independent of the dimension $n$. This estimate shares some similarities with those found in the theory of first- and second-order Hamilton-Jacobi equations. This can be heuristically justified by the presence of the term $(1-2K)|Du|^2$ in the equation of $G_K$, which is concave for $K>1/2$.
\end{rem}

The following is a refinement of an estimate found by R. Hamilton in Theorem 1.1 of \cite{Hamilton} and in Theorem 1.1 of P. Souplet-Q.S. Zhang in \cite{SoupletZhang}. The same estimate of \cite{Hamilton,SoupletZhang} was studied in a more general framework in Theorem 6.46 of \cite{Stroock} using a probabilistic interpretation of the B\"ochner's identity.
\begin{thm}\label{thmHSZ}
Suppose that $\Omega$ satisfies \eqref{intcone} and \eqref{conv}. Let $u$ be a solution to \eqref{heat} such that $u\in L^\infty$ with $0<u\leq A$, and assume \eqref{hypgrad}-\eqref{hypderiv}. Then

\begin{equation}\label{HSZ}
|D\log u(x,\tau)|^2\leq \frac1\tau\log\left(\frac{A}{u(x,\tau)}\right),\ \mbox{ for any }x\in\Omega\mbox{ and for any }\tau\in (0,T).
\end{equation}
Moreover, for any $\delta>0$, we have the integral bound
\begin{align}\label{HSZ2}
\sup_{\rho_\tau\in \mathcal{F}^+_\infty(\Omega)}\iint_{Q_\tau}2t&|D^2\log( u(x,t) +\delta)|^2\rho(x,t)\,dxdt  \\
\leq \sup_{y\in\Omega}&\left(\log\left(\frac{A+\delta}{u(y,\tau)+\delta}\right)-\tau \frac{|Du(y,\tau)|^2}{(u(y,\tau)+\delta)^2}\right)\notag
\end{align}
for any $\tau\in(0,T)$, where $\rho$ is the solution to \eqref{fp} with $b(x,t)=2D\log(u+\delta)$ starting from the datum $\rho(\tau)=\rho_\tau$.\\
If in addition $u\geq\delta_0>0$, \eqref{HSZ2} holds true with $\delta=0$.
\end{thm}
\begin{proof}
For any arbitrary $\delta>0$, we consider $v=\log(u+\delta)$. Recall that
\[
\partial_tv-\Delta v-|Dv|^2=0\text{ in }Q_T.
\]
The estimate is now a consequence of the adjoint-Bernstein method. By Lemma \ref{eqz}-(i) we have the following equation solved by $g=t\omega$, $\omega=|Dv|^2$
\[
\partial_t g-\Delta g+2t|D^2v|^2-2\left\langle Dv, Dg\right\rangle= \omega\text{ in }Q_T.
\]
If we fix $\tau\in(0,T)$, in $Q_\tau$ we can use as a test function the solution $\rho$ of the Cauchy-Neumann problem \eqref{fp} with drift $b(x,t)=2Dv(x,t)$ starting from $\rho(\tau)=\rho_\tau$ with $\rho_\tau\in \mathcal{F}^+_\infty(\Omega)$. We stress that the drift $b(x,t)=2Dv(x,t)=2\frac{Du(x,t)}{u(x,t)+\delta}$ belongs to $L^\infty$. We thus find the following inequality 
\begin{equation}\label{thusineq}
\int_{\Omega}g(\tau)\rho_\tau(x)\,dx+\iint_{Q_\tau}2t|D^2v|^2\rho\,dxdt\leq \iint_{Q_\tau}|Dv|^2\rho\,dxdt.
\end{equation}
As before, we note that the inequality is a consequence of the convexity of $\Omega$ and of Lemma \ref{denu}. When $\Omega=\R^n$ we test against the solution of the Cauchy problem on the whole space, therefore we do not have the boundary term and the above inequality is in fact an identity. We now bound the term $\iint_{Q_\tau}|D\log u|^2\rho\,dxdt$. We test in the weak sense the equation of $v$ by $\rho$ and use the Neumann condition to find
\[
\int_0^\tau\langle\partial_t v,\rho\rangle_{W^{1,2}(\Omega)-W'_\Omega}\,dt+\iint_{Q_\tau}(\left\langle Dv, D\rho\right\rangle-|Dv|^2\rho)\,dxdt=0,
\]
while we test the equation of $\rho$ by $v$ using the Neumann condition to discover
\[
\int_0^\tau-\langle\partial_t \rho,v\rangle_{W^{1,2}(\Omega)-W'_\Omega}\,dt+\iint_{Q_\tau}(\left\langle D\rho, Dv\right\rangle+\mathrm{div}(2Dv\rho))v\,dxdt=0.
\]
We subtract the second identity from the first and obtain, after integrating by parts, 
\[
\int_{\Omega}v(\tau)\rho_\tau(x)\,dx-\int_{\Omega}v(0)\rho(0)\,dx+\iint_{Q_\tau}|Dv|^2\rho\,dxdt=0.
\]
Since the conservation of mass showed in Theorem \ref{well} yields $\int_{\Omega}\rho(0)=\int_{\Omega}\rho(\tau)\,dx=1$, from the previous identity we obtain the following inequality
\begin{multline*}
\iint_{Q_\tau}|Dv|^2\rho\,dxdt=\int_{\Omega}\log(u(0)+\delta)\rho(0)\,dx-\int_{\Omega}\log(u(\tau)+\delta)\rho_\tau(x)\,dx\\
\leq \int_{\Omega}\log(A+\delta)\rho(0)\,dx-\int_{\Omega}\log(u(\tau)+\delta)\rho_\tau(x)\,dx=\int_{\Omega}(\log(A+\delta)-\log(u(\tau)+\delta))\rho_\tau(x)\,dx.
\end{multline*}
This gives us the bound
\begin{equation}\label{crossest}
\iint_{Q_\tau}|Dv|^2\rho\,dxdt\leq \int_{\Omega}\log\left(\frac{A+\delta}{u(\tau)+\delta}\right)\rho_\tau(x)\,dx.
\end{equation}
Inserting \eqref{crossest} in \eqref{thusineq} we deduce
\begin{equation}\label{onetwopunch}
\int_{\Omega}\tau |Dv|^2(x,\tau)\rho_\tau(x)\,dx+\iint_{Q_\tau}2t|D^2v|^2\rho\,dxdt\leq \int_{\Omega}\log\left(\frac{A+\delta}{u(\tau)+\delta}\right)\rho_\tau(x)\,dx.
\end{equation}
On one hand, \eqref{onetwopunch} is saying that
$$
\int_{\Omega}\tau \frac{|Du(x,\tau)|^2}{(u(x,\tau)+\delta)^2}\rho_\tau(x)\,dx\leq \int_{\Omega}\log\left(\frac{A+\delta}{u(x,\tau)+\delta}\right)\rho_\tau(x)\,dx\quad\mbox{ for every }\rho_\tau\in \mathcal{F}^+_\infty(\Omega)
$$
which, by passing to the supremum over $\rho_\tau$, implies
\begin{equation}\label{positivo}
\frac{|Du(x,\tau)|^2}{(u(x,\tau)+\delta)^2}\leq \frac{1}{\tau}\log\left(\frac{A+\delta}{u(x,\tau)+\delta}\right)\quad\mbox{ for every }x\in\Omega.
\end{equation}
By letting $\delta\to 0^+$ in \eqref{positivo} we infer the validity of \eqref{HSZ}. On the other hand, if we keep in mind \eqref{positivo},
we can rewrite \eqref{onetwopunch} as follows
\begin{equation}\label{perdopo}
\iint_{Q_\tau}2t|D^2v|^2\rho\,dxdt\leq \int_{\Omega}\left(\log\left(\frac{A+\delta}{u(x,\tau)+\delta}\right)-\tau \frac{|Du(x,\tau)|^2}{(u(x,\tau)+\delta)^2}\right)\rho_\tau(x)\,dx.
\end{equation}
Again by passing to the supremum over $\rho_\tau$, we obtain \eqref{HSZ2}. Moreover, if $u\geq\delta_0>0$, the same argument carries over with $\delta=0$ and $v=\log u$.
\end{proof}

\begin{rem}
Estimate \eqref{HSZ} was studied in complete Riemannian manifolds with curvature lower bounds in Theorem 6 of \cite{KprocAMS} and Theorem 3.2 of \cite{ZhangIMRN}.
\end{rem}

\begin{rem}
The paper \cite{BakryGentilLedouxSNS} (cf. the discussion below formula (1.8) therein), see also \cite{Stroock}, shows that when $0<u\leq 1$ one has the log-Lipschitz type regularization effect
\[
|D\psi(t)|^2\leq \frac{1}{2t},\ \psi=\sqrt{\log\left(\frac{1}{u}\right)}.
\]
Note that by expanding the term of the left-hand side one ends up with
\[
|D\log u(t)|^2\leq \frac{2}{t}\log\left(\frac{1}{u(t)}\right).
\]
\end{rem}
\begin{rem}
The way of estimating the (crossed) term $\iint |Dv|^2\rho$ in \eqref{crossest} is inspired from the theory of Mean Field Games \cite{LL}. The significance of this estimate is widely treated in \cite{cg20,PorrARMA}, while its importance from the stochastic viewpoint is described in \cite{BKRS,PorrARMA}. \end{rem}

\subsection{Second-order smoothing bounds}\label{sec;second}

The following result is due to P. Li and S.-T. Yau, cf. Theorem 1.1 in \cite{LiYauActa}, where the authors investigated second-order estimates for solutions of the heat equation on compact manifolds with convex boundary  (see also Section 4.2 in \cite{EvansSurvey} or \cite{EvansNotes}). Here we present a different new proof by duality following the lines of the previous sections. 

\begin{thm}\label{liyauconvex}
Suppose that $\Omega$ satisfies \eqref{intcone} and \eqref{conv}. Let $u$ be a positive solution to \eqref{heat} such that $u$ satisfies \eqref{hypgrad}-\eqref{hypderiv}. Then we have
\[
\Delta(\log u(x,\tau))=\frac{\partial_t u(x,\tau)}{u(x,\tau)}-\frac{|Du(x,\tau)|^2}{u^2(x,\tau)}\geq -\frac{n}{2\tau},\ \mbox{ for any }x\in\Omega\mbox{ and for any }\tau\in (0,T).
\]
\end{thm}

\begin{proof}
Fix an arbitrary $\delta>0$, and consider $v=\log (u+\delta)$. Lemma \ref{eqz}-(ii) provides the following equation satisfied in $Q_T$ by $z=t^2\Delta v$
\[
\partial_t z-\Delta z-2|D^2v|^2t^2-2\left\langle Dv, Dz\right\rangle=2t\Delta v.
\]
For any $\tau\in (0,T)$, in $Q_\tau$ we consider as before the solution $\rho$ of the Cauchy-Neumann problem \eqref{fp} with $L^\infty$-drift $b(x,t)=2Dv(x,t)=2\frac{Du(x,t)}{u(x,t)+\delta}$ starting from $\rho(\tau)=\rho_\tau$ with $\rho_\tau\in \mathcal{F}^+_\infty(\Omega)$. 
By exploiting the viscous Hamilton-Jacobi equation satisfied by $v$ and the convexity of $\Omega$ via Lemma \ref{denu}, we notice that
\[
\partial_\nu z=t^2\partial_\nu (\Delta v)=t^2\partial_\nu(\partial_t v-|Dv|^2)= t^2\partial_t (\partial_\nu v)-t^2\partial_\nu(|Dv|^2)\geq 0 \text{ on }\Sigma_\tau.
\]
By using the latter information together with Cauchy-Schwarz inequality, we obtain by duality that
\[
\int_{\Omega}z(\tau)\rho_\tau(x)\,dx-2\iint_{Q_\tau}|D^2v|^2t^2\rho\,dxdt\geq -2\iint_{Q_\tau}t|\Delta v| \rho\geq -2\sqrt{n}\iint_{Q_\tau}t|D^2v|\rho\,dxdt.
\]
We can now use Young's inequality and the conservation of mass constraint $\int_\Omega \rho(t)\,dx=1$ ensured by Theorem \ref{well} in order to infer that
\begin{align*}
-2\sqrt{n}\iint_{Q_\tau}t|D^2v|\rho\,dxdt &\geq -2\iint_{Q_\tau}t^2|D^2v|^2\rho\,dxdt-\frac{n}{2}\iint_{Q_\tau}\rho\,dxdt\\
&= -2\iint_{Q_\tau}t^2|D^2v|^2\rho\,dxdt-\frac{n}{2}\tau.
\end{align*}
The combination of the previous two inequalities yields
$$
\tau^2\int_{\Omega}\Delta v(x,\tau)\rho_\tau(x)\,dx=\int_{\Omega}z(\tau)\rho_\tau(x)\,dx\geq -\frac{n}{2}\tau.
$$
Passing to the supremum over $\rho_\tau\in \mathcal{F}^+_\infty(\Omega)$, we conclude the pointwise bound 
\[
\Delta (\log u(\tau)+\delta)\geq -\frac{n}{2\tau}.
\]
Letting $\delta\to 0^+$ we reach the Li-Yau inequality. 
\end{proof}
\begin{rem}
A generalization of the Li-Yau estimate on unbounded domains can be found in Theorem 12 of \cite{BakryQian}: it treats heat flows on complete manifolds with convex boundary under a lower bound on $\partial_tu$. Theorem 1.3 of \cite{WangPJM} considered the Li-Yau inequality with bounds from below on the Ricci curvature on convex and nonconvex unbounded domains via a generalized form of the maximum principle.
\end{rem}
We conclude with an upper second derivative bound on the Laplacian of the logarithm (i.e. a semi-log-super\-harmonic estimate following the terminology of \cite{Lions82Book}, cf. Definition \ref{defisemi}) using the new integral estimate \eqref{HSZ2} in Theorem \ref{thmHSZ}. This appeared by different methods based on the maximum principle in Theorem E.36 of \cite{Ricciflowbook} and in Theorem 1.1 of \cite{HanZhang}. A related estimate was addressed also in Theorem 5.1 of \cite{LeeVazquez}, where the authors proved a semi-log-superharmonic estimate when the initial datum is compactly supported via heat kernel estimates. Recall that Theorem \ref{liyauconvex} provided the semi-log-subharmonic estimate
\[
\Delta \log u(\tau)\geq -\frac{n}{2\tau},
\]
whose constant is independent of the sup-norm of the solution and the initial datum.
\begin{thm}\label{HSZ2nd}
Suppose that $\Omega$ satisfies \eqref{intcone} and \eqref{conv}. Let $u$ be a solution to \eqref{heat} such that $u\in L^\infty$ with $0<u\leq A$, and assume \eqref{hypgrad}-\eqref{hypderiv}. Then, for any $\tau\in (0,T)$, we have
\begin{equation}\label{so1}
\Delta \log u(\tau)+2|D\log u(\tau)|^2=\frac{\Delta u(\tau)}{u(\tau)}+\frac{|D u(\tau)|^2}{u^2(\tau)}\leq \frac{1}{\tau}\left(n+\frac{7}{2}\log\left(\frac{A}{u(\tau)}\right)\right).
\end{equation}
\end{thm}
\begin{proof}
Fix arbitrary $\delta>0$ and $\rho_\tau\in \mathcal{F}^+_\infty(\Omega)$, and let $v=\log (u+\delta)$. In the proof of Theorem \ref{thmHSZ} we established the integral bound \eqref{perdopo}, from which we deduce
\begin{equation}\label{perqui}
\iint_{Q_\tau} t|D^2v|^2\rho\,dxdt\leq \frac12\int_{\Omega}\log\left(\frac{A+\delta}{u(\tau)+\delta}\right)\rho_\tau(x)\,dx.
\end{equation}
On the other hand, we can use Lemma \ref{eqz}-$(ii)$ to find the following equation for $w=t^2\Delta v$
\[
\partial_t w-\Delta w-2\left\langle Dv, Dw\right\rangle=2t^2|D^2v|^2+2t\Delta v.
\]
Exploiting the convexity of the domain (as in the proof of Theorem \ref{liyauconvex}), this implies by duality
\[
\int_{\Omega}w(\tau)\rho_\tau(x)\,dx\leq 2\iint_{Q_\tau}t^2|D^2v|^2\rho\, dxdt+\iint_{Q_\tau}2t\Delta v\rho\,dxdt.
\]
We then have by the Cauchy-Schwarz and Young inequalities
\[
\iint_{Q_\tau}2t\Delta v\,dxdt\leq 2\sqrt{n}\iint_{Q_\tau}t|D^2 v|\rho\,dxdt\leq \iint_{Q_\tau}t^2|D^2v|^2\rho\, dxdt+n\iint_{Q_\tau}\rho\,dxdt.
\]
Therefore, since $t\leq \tau$ in $Q_\tau$ and using the conservation of mass property from Theorem \ref{well}, we obtain
\begin{align*}
\int_{\Omega}w(\tau)\rho_\tau(x)\,dx&\leq 3\iint_{Q_\tau}t^2|D^2v|^2\rho\, dxdt+n\int_0^\tau\int_\Omega \rho_\tau(x)\,dxdt \\
&\leq 3\tau\iint_{Q_\tau}t|D^2v|^2\rho\, dxdt+n\tau\int_\Omega \rho_\tau(x)\,dx.
\end{align*}
Inserting in the latter inequality the bound in \eqref{perqui}, we infer
\[
\tau^2\int_{\Omega}\Delta v(x,t)\rho_\tau(x)\,dx=\int_{\Omega}w(\tau)\rho_\tau(x)\,dx\leq \tau\int_{\Omega}\left(\frac32\log\left(\frac{A+\delta}{u(\tau)+\delta}\right)+n\right)\rho_\tau(x)\,dx.
\]
The arbitrariness of $\rho_\tau\in \mathcal{F}^+_\infty(\Omega)$ thus yields 
$$
\Delta \log (u(\tau)+\delta)\leq \frac{3}{2\tau}\log\left(\frac{A+\delta}{u(\tau)+\delta}\right)+\frac{n}{\tau}.
$$
By letting $\delta\to 0^+$ we finally have
$$
\Delta \log (u(\tau))\leq \frac{3}{2\tau}\log\left(\frac{A}{u(\tau)}\right)+\frac{n}{\tau},
$$
which shows the validity of \eqref{so1} once we combine it with \eqref{HSZ}.
\end{proof}

\section{Geometric preserving properties for the heat flow in $\R^n$}\label{sec;conservation}

In this final section we restrict ourselves to the case $\Omega=\R^n$. We shall use the notation
$$
S_T=\R^n\times(0,T).
$$
As in the subSection \ref{warm} and differently to the case of the previous section, we deal with bounds for the full Hessian matrix of the solution. It is known that matrix estimates of Hamilton type are delicate in the case of domains with boundaries and they hold true only in special situations, see e.g. \cite{Pule}. We address the conservation of geometric properties of the solution $u$ to the classical Cauchy problem on $\R^n$ starting from an initial datum $u(0)$. We recall that convexity/concavity are geometric properties preserved by the heat equation: indeed, any pure second derivative $u_{ee}$ is still a solution of the heat equation, therefore the conclusion follows by the maximum principle. Nonetheless, an inspection of the explicit fundamental solution of the heat equation suggests that even the $\log$-concavity/convexity might be preserved. \\

We will assume
\begin{equation}\label{pres1}
u(0)\geq \delta>0\text{ and }Du(0),D^2u(0)\in L^\infty(\R^n)
\end{equation}
and
\begin{equation}\label{pres2}
u(0)\in W^{3,2}(\R^n).
\end{equation}
As discussed in subSection \ref{sec21}, the latter assumption \eqref{pres2} ensures higher-order integrability of $u$ that makes rigorous our next variational computations. Assumption \eqref{pres1} guarantees, among others, the boundedness condition of the velocity field of the adjoint equation.\\

We start with the following
\begin{defn}\label{defisemi}
We say that a twice-differentiable function $u$ is semi-log-concave (resp. semi-log-convex) if $\log u$ is semiconcave (semiconvex), namely $D^2\log u\leq C\mathbb{I}_n$ (resp. $D^2\log u\geq -C\mathbb{I}_n$) for a positive constant $C$. We say that $u\in C^2$ is semi-log-superharmonic (resp. semi-log-subharmonic) if $\log u$ is semi-superharmonic (semi-subharmonic), namely $\Delta \log u\leq C$ (resp. $\Delta \log u\geq -C$).
\end{defn}
 
The main contributions of this section with respect to the current literature, other than a new method of proof, are the treatment of semiconcave/semiconvex terms as well as PDEs with integral data. We finally devote particular attention to provide explicit estimates on the magnitude on the norms involved.

\begin{thm}\label{semilogconvex}
Let $u$ be a positive solution to the heat equation in $S_T$ satisfying \eqref{pres1}-\eqref{pres2}. Suppose that $u(0)$ is semi-log-convex at the initial time. Then $u$ is semi-log-convex for positive times with the same constant.
\end{thm}

\begin{proof}
Lemma \ref{eqz}-$(ii)$ with $\chi=1$ gives the following equation satisfied by $h=\langle D^2v \xi,\xi\rangle$ where $v=\log u$,
\[
\partial_t h-\Delta h-2|D^2v\xi|^2-2\left\langle Dv, Dh\right\rangle=0.
\]
By duality, testing the above equation against the solution $\rho$ of \eqref{fp} with $b(x,t)=2Dv$, we get
\[
\int_{\R^n}h(\tau)\rho_\tau(x)\,dx-\int_{\R^n}h(0)\rho(0)\,dx-2\iint_{S_\tau}|D^2v\xi|^2\rho\,dxdt=0.
\]
Since $\rho(0)\geq0$, this implies that
$$
\int_{\R^n}\left\langle D^2v(x,\tau)\xi,\xi\right\rangle\rho_\tau(x)\,dx\geq \int_{\R^n}h(0)\rho(0)\,dx=\int_{\R^n}\left\langle D^2v(x,0)\xi,\xi\right\rangle\rho(0)\,dx\geq-c_0.
$$
Thus $D^2(\log u)(\tau)\geq-c_0\mathbb{I}_n$, i.e. $(\log u)(\tau)$ is semiconvex, $\tau>0$, and hence the assertion. \end{proof}
As a partial converse, we have the following property: it shows the preservation of semi-log-concavity for positive solutions of the heat equation. The proof needs an intermediate two-side second derivative estimate for $\log$ solutions of the heat equation.
\begin{thm}\label{BL}
Let $u$ be a positive solution to the heat equation in $S_T$ satisfying \eqref{pres1}-\eqref{pres2}. Suppose that $D^2\log u(0)\leq c_0\mathbb{I}_n$ in $\R^n$. Assume that 
\begin{equation}\label{bound}
\|D\log u(0)\|_{L^\infty(\R^n)}\leq K_0.
\end{equation}
Then $u$ is semi-log-concave for positive times (with a different constant). 
\end{thm}
\begin{proof}
We first prove that for $v=\log u$ the following two-side estimate holds
\[
\iint_{S_\tau}|D^2v|^2\rho\,dxdt\leq K^2_0.
\]
Set $\omega=|Dv|^2$ and obtain the identity
\[
\partial_t \omega-\Delta \omega+2|D^2v|^2-2\left\langle Dv, D\omega\right\rangle=0.
\]
By duality and using $\int_{\R^n}\rho(t)\,dx=1$ for all $t\in[0,T]$ we have
\[
2\iint_{S_\tau}|D^2v|^2\rho\,dxdt\leq -\int_{\R^n}\omega(\tau)\rho_\tau(x)\,dx+\int_{\R^n}\omega(0)\rho(0)\,dx\leq K_0^2.
\]
We now use that $w=\partial_{ee} v$, $e$ being a unitary direction of $\R^n$, solves
\[
\partial_t w-\Delta w-2|Dv_e|^2-2\left\langle Dv, Dw\right\rangle=0.
\]
By duality this yields
\[
\int_{\R^n}w(\tau)\rho_\tau(x)\,dx=2\iint_{S_\tau}|Dv_e|^2\rho\,dxdt+\int_{\R^n}w(0)\rho(0)\,dx\leq K_0^2+c_0.
\]
\end{proof}

\begin{cor}\label{bl}
We have the following two-side estimate for positive solutions of the heat equation satisfying \eqref{pres1}-\eqref{pres2} and \eqref{bound}, where $u(0)$ is additionally both semi-log-concave and semi-log-convex (with constant $c_0$)
\[
-c_0\mathbb{I}_n\leq D^2\log u(\tau)\leq \left(K_0^2+c_0\right)\mathbb{I}_n.
\]
\end{cor}
\begin{rem}
A discussion about the dependence of the log-concavity/convexity estimates in Corollary \ref{bl}, along with a comparison with \cite{BrascampLieb}, are in order. While the semi-log-convexity is preserved by the heat flow with the same constant, cf. Theorem \ref{semilogconvex}, this does not occur in the log-concavity bound. In fact, this depends on the Lipschitz regularity of the solution through the constant $K_0$ and the semi-log-concavity of the initial datum. This quantitative bound is close in spirit to \cite[Theorem VII.3]{IshiiLions} or \cite[Theorem 3.1]{Gigaetal}, where solutions or initial data are assumed Lipschitz continuous. The classical proof via the Prekopa-Leindler inequality of \cite{BrascampLieb} (cf. \cite[Lemma 3.2]{Bez}) shows that if $\log u(0)$ is concave, then $\log u(t)$ is concave for $t>0$ without any Lipschitz regularity requirement, but requires a log-concave (not semiconcave) initial datum. An alternative approach to \cite{BrascampLieb} can be also found in the theory of stochastic control, see \cite[Lemma IV.10.6]{FlemingSoner}. \end{rem}
We conclude with semi-log-concavity/convexity properties for solutions of heat equations with space-time potentials.
\begin{thm}
Let $u$ be a solution to the heat equation with potential
\[
\partial_t u-\Delta u+F(x,t)u=0\text{ in }S_T
\]
and satisfying \eqref{pres1}-\eqref{pres2}.
\begin{itemize}
\item[(a)] Assume that $D^2\log u(0)\geq -c_1\mathbb{I}_n$ along with $D^2 F\leq c_{f,1}(t)\mathbb{I}_n$, $c_{f,1}\in L^1(0,T)$. Then
\[
D^2\log u\geq -c_1\mathbb{I}_n-\mathbb{I}_n\int_0^\tau c_{f,1}(t)\,dt.
\]
\item[(b)] If instead $\Delta\log u(0)\geq -c_2$ along with $\Delta F\leq c_{f,2}(t)\in L^1(0,T)$, then
\[
\Delta \log u\geq -c_2-\int_0^\tau c_{f,2}(t)\,dt.
\]
\item[(c)] If $D^2\log u(0)\leq c_3\mathbb{I}_n$ and $DF\in L^1(0,\tau;L^\infty(\R^n))$ with $D^2 F\geq -c_{f,3}(t)\mathbb{I}_n$, $c_{f,3}\in L^1(0,T)$. Assume that for some positive constants $K_0,K>0$
\[
\|D\log u(0)\|_{L^\infty(\R^n)}\leq K_0\text{ and }\|D\log u\|_{L^\infty(S_T)}\leq K.
\]
Then
\[
D^2 \log u \leq (K_0^2+K\|DF\|_{L^1(0,\tau;L^\infty(\R^n))})\mathbb{I}_n+c_3\mathbb{I}_n+\mathbb{I}_n\int_0^\tau c_{f,3}(t)\,dt.
\]
\end{itemize}
\end{thm}
\begin{rem}
Part (c) agrees, providing a quantitative statement, with Theorem 6.1 in \cite{BrascampLieb} valid for the heat equation with convex potential (i.e. when $c_{f,3}=0$), see also \cite{Borell} and Theorem 1.13 of \cite{BL75}. Note that the results in \cite{BL75,BrascampLieb} apply for $F=F(x)$ convex with an addition of a quadratic potential, show the preservation of log-concavity instead of semiconcavity, and do not provide an explicit estimate on the size of the norm. We also refer to Theorem 2.1 and 2.2 of \cite{YauCAG} for related quantitative results.
\end{rem}

\begin{proof}
We again proceed with the transformation $v=\log u$ to find
\[
\partial_t v-\Delta v-|Dv|^2=-F(x,t).
\]
We show only the second estimate, the more general on the quadratic form $h$ being similar. We set $w=\Delta v$ and discover that by Lemma \ref{eqz}-(ii) with $\chi=1$ the function $w$ solves
\[
\partial_t w-\Delta w-2|D^2v|^2-2\left\langle Dv, Dw\right\rangle=-\Delta F.
\]
By duality, using the conservation of mass for the adjoint problem and the fact that $\rho\geq0$, we get
\begin{multline*}
\int_{\R^n}w(\tau)\rho(\tau)\,dx\\
=2\iint_{Q_\tau}|D^2v|^2\rho\,dxdt+\int_{\R^n}w(0)\rho(0)-\iint_{Q_\tau} (\Delta F)\rho\,dxdt \geq -c_2-\int_0^\tau c_{f,2}(t)\,dt.
\end{multline*}
This implies that
\[
\Delta\log u(\tau)\geq -c_{2}-\int_0^\tau c_{f,2}(t)\,dt.
\]
The estimate on the Hessian is similar (see subSection \ref{warm}).\\
We now discuss the semi-log-concavity estimate in the presence of a potential. This requires an estimate on the second-order term as in Theorem \ref{BL}, as we are proving an estimate in the opposite direction. As in the first part of Theorem \ref{BL} we claim that
\begin{equation}\label{K0K}
2\iint_{S_\tau}|D^2v|^2\rho\,dxdt\leq K_0^2+K\|DF\|_{L^1(0,\tau;L^\infty(\R^n))}.
\end{equation}
Indeed, $\omega=|Dv|^2$ solves
\[
\partial_t\omega-\Delta \omega+2|D^2v|^2-2\left\langle Dv, D\omega\right\rangle=-\left\langle DF, Dv\right\rangle.
\]
By duality, this implies
\[
2\iint_{S_\tau}|D^2v|^2\rho\,dxdt=\int_{\R^n}\omega(0)\rho(0)\,dx-\int_{\R^n}\omega(\tau)\rho(\tau)\,dx-\iint_{S_\tau}\left\langle DF, Dv\right\rangle\rho\,dxdt.
\]
Then \eqref{K0K} follows. Therefore, the equation solved by $w=v_{ee}$ reads
\[
\partial_t w-\Delta w+2|Dv_e|^2-2\left\langle Dv, Dw\right\rangle=-\partial_{ee}F.
\]
By duality and using \eqref{K0K} this yields
\begin{multline*}
\int_{\R^n}w(\tau)\rho_\tau(x)\,dx=-2\iint_{S_\tau}|Dv_e|^2\rho\,dxdt-\iint_{S_\tau}F_{ee}\rho\,dxdt+\int_{\R^n}w(0)\rho(0)\,dx\\
\leq K_0^2+K\|DF\|_{L^1(0,\tau;L^\infty(\R^n))}+c_{3}+\int_0^\tau c_{f,3}(t)\,dt.
\end{multline*}

\end{proof}


\end{document}